\documentclass[12pt,a4paper,oneside,onecolumn]{article}

\usepackage{libertine}
\usepackage{courier}
\usepackage[english]{babel}

\usepackage{amsthm,amsmath,amssymb,mathtools}
\usepackage{thmtools}

\numberwithin{equation}{section} 

\usepackage{titlesec}
\usepackage{float}
\usepackage{graphicx}
\usepackage{comment}
\usepackage{nicefrac}
\usepackage{datetime}
\usepackage{xcolor}
\usepackage[shortlabels]{enumitem}
\usepackage{cancel}
\usepackage{titling}
\usepackage[normalem]{ulem} 

\usepackage{hyperref}
\usepackage{bookmark}
\usepackage{cleveref} 
\crefname{prop}{proposition}{propositions}
\Crefname{prop}{Proposition}{Propositions}
\crefname{cor}{corollary}{corollaries}
\Crefname{cor}{Corollary}{Corollaries}

\theoremstyle{plain}
\newtheorem{theorem}{Theorem}[section]
\newtheorem{lemma}[theorem]{Lemma}
\newtheorem{cor}[theorem]{Corollary}
\newtheorem{conj}[theorem]{Conjecture}
\newtheorem{prop}[theorem]{Proposition}
\theoremstyle{definition}

\newtheorem{question}[theorem]{Question}
\theoremstyle{remark}
\newtheorem*{remark*}{Remark}

\newcommand{\R}{\mathbb{R}}
\newcommand{\Id}{\mathcal{I}^2\left( \left[0,2\pi\right],\mathbb{R}^d\right)}

\newcommand{\Ind}{\mathcal{I}^n\left( \left[0,2\pi\right],\mathbb{R}^d\right)}
\newcommand{\In}{\mathcal{I}^{n}\left(\left[0,2\pi\right],\mathbb{R}\right)}

\DeclareMathOperator{\diff}{\operatorname{Diff}_{H^n}\big(\left[0,2\pi\right]\big)}

\DeclareMathOperator{\len}{\mathrm{len}}
\newcommand{\dist}{\mathrm{dist}_G}

\title{On the Metric Completion of Immersed Open Curves}
\author{Ronny Gelman\thanks{Department of Mathematics, University of Toronto} \and Cy Maor\thanks{Einstein Institute of Mathematics, Hebrew University of Jerusalem}}
\date{}

\begin{document}

\maketitle

\vspace{-1.2cm}
\begin{abstract}

The completeness properties of spaces of immersed curves endowed with Sobolev-type reparametrization-invariant Riemannian metrics have been studied extensively in recent years. 
Whereas constant-coefficient metrics of large enough order on the space of closed curves are known to be complete, this is not the case for open curves.
In this work, we characterize the metric completion of the space of real-valued immersed open curves with respect to these metrics. 
We show that the set of additional limit points in the completion is homeomorphic to $H^n$-diffeomorphisms of a closed interval. 
This result answers a question of the structure of the completion, posed in \cite{BMM23}, in the negative, and suggests a more refined conjecture on the structure in the case of multi-dimensional ambient space.

\end{abstract}
\section{Introduction}
\label{chap:intro}
\subsection{background}

Reparametrization-invariant Sobolev-type Riemannian metrics on spaces of immersed curves have attracted significant interest throughout the years.
From a theoretical point of view, these are the natural generalizations of right-invariant metrics on $\operatorname{Diff}(S^1)$, which are related to many equations in hydrodynamics \cite{bauer2014overview}.
From an applications point of view, these spaces play a significant role in shape analysis, where these metrics are used to compare objects whose shapes are described by these immersions, which can be the outline of a planar shape (closed curves) or time trajectories (open curves) \cite{Bauer_2014,srivastava2016functional,su2014statistical,su2018comparing,younes2010shapes}.

Since these are infinite dimensional Riemannian manifolds, many ``elementary'' properties of Riemannian metrics, like the non-degeneracy of their geodesic distance, the existence of geodesics and the completeness of the manifolds, are highly non-trivial, and strongly depends on the order of the metric --- the number of derivatives of the tangent vectors it involves.
Generally speaking, the higher the order of the metric the better the space ``behaves''.
This was studies in a series of works \cite{bauer2024completeness,BMM23,bruveris2015completenesspropertiessobolevmetrics,bruveris2014geodesic,bruveris2017completeness,michor2006riemannian}, mostly concerned with close curves.\footnote{Other extensions of this line of works are for embedded curves \cite{dohrer2025complete}, and for spaces of immersed closed surfaces \cite{bauer2025completeness}}
As shown in \cite{bauer2024completeness}, the exact threshold for completeness is $H^{3/2}$, that is, for any $s>3/2$, the geodesic equations of the Sobolev metric of order $s$ is globally well-posed.
This is proven by showing that the space of Sobolev immersions of order $H^s$ are metrically-complete, and using the only part of the Hopf--Rinow theorem that holds in infinite dimensions.

Open curves, on the other hand, exhibit a different behavior:
Spaces of immersed open curves endowed with constant coefficient Sobolev-type reparametrization-invariant metrics of any order have been shown to be metrically incomplete: examples in \cite{bauer2019relaxed,BMM23} construct paths of immersions, with diminishing curve length, that leave the space in finite time.

The analysis in \cite{BMM23} does, however, establish that Cauchy sequences of open-curve immersions whose lengths are bounded away from zero, converge. 
This implies that all possible instances of metric incompleteness can be traced to sequences of curves with vanishing curve length; moreover, in \cite{BMM23} it was shown that various such sequences represented the same point in the completion space, which led to the following question:
\begin{question}[\cite{BMM23}]
Let $G$ be a constant coefficient Sobolev-type reparametrization-invariant Riemannian metric of order $n\ge 2$ on the space of immersed open curves of regularity $H^n$.
Does the metric completion of this space consist of a single additional point representing the limit of all vanishing-lenght Cauchy sequences?
\end{question}

In this paper, we focus on the simplest setting for immersions of open curves --- namely, real-valued immersions $[0,2\pi]\to \R$. 
We endow this space with  constant coefficient Sobolev-type reparametrization-invariant Riemannian metrics of order $n\ge 2$, with a non-zero curvature term. 
The choice to study real-valued immersions is a natural one, as the only existing examples of diverging Cauchy sequences in \cite{bauer2019relaxed,BMM23} are constructed using straight-line curves.

In these settings, we answer the above-mentioned question in the negative by proving the existence of a family of distinct limit points. 
Moreover, we identify the structure of the set of additional limit points and prove that it is homeomorphic to $\diff$ endowed with its natural topology. 
In view of this results, we suggest an updated conjecture regarding the metric completion in the more general case. 

Our main result, the updated conjecture, and a brief overview of the main proof are discussed in the following subsection, followed by a preliminary section establishing terminology and tools. The main results are proved in \Cref{sec: results}. 

\subsection{Main result}
\label{subsec: results presentation}

To state the results, we briefly describe our objects of interest. Consider the space of  Euclidean-valued Sobolev maps $H^n\left((0,2\pi),\R^d\right)$, where $n\ge2$ and $d\ge1$.
We define the space of $H^n$-regular immersed open curves as 
\begin{equation}\label{eq:Ind}
\Ind\coloneqq \bigg\{c\in H^n\left(\left(0,2\pi\right),\R^d\right) ~:~ \forall \theta\in [0,2\pi], |c'(\theta)|>0\bigg\},
\end{equation}
which is well defined by Sobolev embedding $H^n(0,2\pi)\subset C^1([0,2\pi])$, which also shows that this is an open set of $H^n\left(\left(0,2\pi\right),\R^d\right)$. 
Thus, it inherits a Hilbert-manifold structure of the Hilbert space $H^n\left(\left(0,2\pi\right),\R^d\right)$.\
We endow the manifold with a constant coefficient, reparametrization-invariant Sobolev-type Riemannian metric $G$ (see \eqref{eq: metric} for the exact definition).
Denote the induced geodesic distance by $\dist$. As previously mentioned, \cite{bauer2019relaxed,BMM23} demonstrate the metric incompleteness of $\Ind$ with respect to $G$.
We denote by $\overline{\Ind}$ its metric completion, and refer to the set of additional limit points by
\[\overline{\Ind}\setminus\Ind.\]

We also denote by $\diff$ the group of $H^n$-regular diffeomorphisms of $[0,2\pi]$, and consider it with its natural topology induced by $\|\cdot\|_{H^n}$.

Our main result regards the completion in the case of $d=1$: 
\begin{theorem}
\label{thm: main}
      Let $n\ge2$, and let $G$ be a constant coefficient, reparametrization-invariant Sobolev-type Riemannian metric of order $n$, with a non-trivial $\dot{H}^2$ term (see \eqref{eq: metric}). Then the mapping
      \[\diff\rightarrow\overline\In\setminus\In,\] 
     taking a diffeomorphism $\varphi\in\diff$ to the limit point of the vanishing-length sequence of immersions $\left(\frac{1}{n} \varphi\right)_{n=1}^\infty$ in $\overline{\In}$, is a homeomorphism. As a result, we can write the metric completion as 
     \[\In\sqcup \diff.\]
     Furthermore, the inverse of the mapping above is Lipschitz continuous when $\diff$ is endowed with the distance induced by its natural embedding into $\In$. 
\end{theorem}

We suspect that the same result extends to $\Ind$, for any $d\ge 1$, possibly even when endowed with metrics that do not explicitly include a second-order term:

\begin{conj}
    \label{conj: main update}
    Let $n\ge2$, $d\ge1$, and $G$ a Riemannian metric as in \eqref{eq: metric} with, possibly, $a_2=0$. Then the metric completion of $\left(\Ind,\dist\right)$ consist of 
    \[\Ind\sqcup \diff.\] 
    Consequently, the metric completion of the quotient space, 
    \[\Ind\slash\diff,\]
    consists of a single additional point, representing the limit of all asymptotically vanishing sequences. 
\end{conj}

We also note that \Cref{thm: main} seems to be one of the few explicit characterizations \cite{clarke2013ricci,darvas2017mabuchi,guedj2014metric,khesin2004flow} of a non-trivial metric completion of an infinite-dimensional manifold endowed with a natural Riemannian structure.

The paper is organized as follows:
in \Cref{chap: Preliminaries} we discuss the space of immersions of open curves, the Riemannian metrics we consider, and some of their basic properties.
In \Cref{subsec: estimates} we establish some estimates in this space, culminating in \Cref{lem: Constant dist inequality}, which serves as the primary tool for the subsequent arguments. 
This lemma demonstrates that the geodesic distance between a pair of arbitrarily short immersions is controlled from below by the geodesic distance between the associated pair of diffeomorphisms in $\diff$. 
The proof of \Cref{thm: main} is then concluded in \Cref{subsec: results}.

\paragraph{Acknowledgements} This work is based on, and extends parts of the Master's thesis of RG at the Hebrew University of Jerusalem. 
CM was partially supported by ISF grant 2304/24 and BSF grant 2022076.
AI was not used in the derivation of the results of this paper, nor in its writing.
\section{Preliminaries}
\label{chap: Preliminaries}
In this section we define spaces of Euclidean-valued immersed open curves and present the tools we will use throughout the analysis that follows. 

Consider the space $\Ind$ as defined in \eqref{eq:Ind}.
As mentioned earlier, it is an open subset of $H^n((0,2\pi),\R^d)$, and thus 
\[T_c\Ind\cong H^n((0,2\pi),\R^d),\]
where we view a tangent vector $h\in T_c\Ind$ as a vector field along $c$. 

We define reparametrization-invariant Sobolev-type Riemannian metrics with constant coefficients on $\Ind$ by
\begin{equation}
\label{eq: metric}
 G_c(h,k)  \coloneqq \sum_{i=1} ^n  a_i \intop_0 ^{2\pi} \langle \partial_s ^i h,\partial_s ^i k\rangle_{\mathbb{R}^d}\ ds,\quad a_i\ge0, \ a_0,a_2,a_n>0,
\end{equation}
where $ \partial_s h=\frac{1}{|c'|}\partial_\theta h \ $  and $ ds=|c'|d\theta$.

For a fixed base curve $c$, and any $h\in T_c\Ind$, we denote $\|\partial_s^i h\|^2_{L^2(ds)}=\int_0^{2\pi}\langle \partial_s ^i h,\partial_s ^i k\rangle_{\mathbb{R}^d}ds$.
The length of a piecewise $C^1$-curve $\gamma:[a,b]\rightarrow M$, is defined as 
    \begin{equation}
    \label{eq: len definition}
    \len(\gamma)\coloneqq\intop_a^b\sqrt{G_{\gamma(t)}\big{(}\gamma_t(t),\gamma_t(t)\big{)}}dt.
   \end{equation}
We then denote the induced geodesic distance by 
    \begin{equation}
    \dist\left(c_1,c_2\right)\coloneqq\inf\len(\gamma),
    \end{equation}
where the infimum is taken over all piecewise $C^1$ paths between $c_1$ and $c_2$.

We denote by $\diff$ the set of $H^n$-diffeomorphisms mapping $[0,2\pi]$ onto itself.
It is a half-Lie group: a smooth manifold which is a topological group, in which right translations are smooth \cite{bauer2025completeness}. 
Moreover, $\diff$ acts smoothly from the right on $\Ind$ by reparametrization.

The Riemannian metrics defined in \Cref{eq: metric} are reparametrization invariant. By this we mean that, 
\[G_c\left (h,k\right)=G_{c\circ\varphi}\left(h\circ\varphi,k\circ\varphi\right)
\]
for any $h,k\in T_c\Ind$ and any $\varphi\in\diff$. This immediately implies that the induced distance is reparametrization invariant in the sense that
\begin{equation}
    \label{eq: rep. invariance}
    \dist\left(c_1,c_2\right)=\dist\left(c_1\circ\varphi,c_2\circ\varphi\right)
\end{equation}
for any pair $c_1,c_2\in\Ind$ and any $\varphi\in\diff$.

We note that for any $c\in\Ind$, the metric $G_c$ as in \eqref{eq: metric} is equivalent to the non-invariant metric induces on $\Ind$ as a subset of the Hilbert space $H^n((0,2\pi),\R^d)$.
As such, $G$ is a strong Riemannian metric (i.e., $G_c$ induces the natural topology on the tangent space for any $c\in \Ind$), and thus its geodesic distance induces the manifold topology:
\begin{lemma}[{\cite[VII, prop.~6.1]{lang2012fundamentals}}]
    \label{lemm: topo equivalence}
    Given a Riemannian manifold $(X,g)$, where $g$ is a strong Riemannian metric,  the topology induced by the geodesic distance $\text{dist}_g$ is the underlying topology obtained from the manifold structure.
\end{lemma}

For an immersion $c\in\Ind$, we refer to the curve length of $c$,  as 
\begin{equation}
    \label{eq: ellc def}
\ell_c\coloneqq\int_0^{2\pi}ds.
\end{equation}
 By reparametrization, it can be verified that  $\|\partial_s^i h\|_{L^2(ds)}=\frac{1}{\ell_c^{(2n-3)\slash2}}\|h\|_{\dot{H}^i(d\theta)}$. Furthermore, we note the following Sobolev-interpolation estimate: 
\begin{lemma}[{\cite[Lemma~4.1]{BMM23}}]
    \label{lemm: norm est}
    For $n\ge2$, $c\in\Ind$,  $h\in T_c\Ind$, and any $0\le k\le n$,  there exists $C=C(k,n)$ such that 
    \begin{equation}
        \ell_c^{2k}\|\partial_s^k h\|^2_{L^2(ds)}\le C\left(\| h\|^2_{L^2(ds)}+\ell_c^{2n}\|\partial_s^n h\|^2_{L^2(ds)}\right).
    \end{equation}
\end{lemma}

We note that for closed curves it is often more standard to consider the metric 
\begin{equation}
    \label{eq: alternative metric def}
    \tilde{G}_c\left(h,k\right) \coloneqq  \intop_0 ^{2\pi} \langle  h, k\rangle_{\mathbb{R}^d}+ \langle \partial_s^n h,\partial_s^n k\rangle_{\mathbb{R}^d} ds.
\end{equation}
Although $G_c$ and $\tilde{G}_c$ induce equivalent norms on any $T_c\Ind$ for a fixed $c\in\Ind$, the constants involved depend on $\ell_c$, and thus do not necessarily induce equivalent distance functions.
In contrast, for closed curves \Cref{lemm: norm est} can be improved \cite{BMM23,bruveris2014geodesic}, and all the constant coefficients metrics of the same order are equivalent, with constants independent of $c$.

Another important distinction between spaces of closed curves (or spaces with length-weighted coefficient metrics) and spaces of open curves is the boundedness of the map $c\mapsto\ell_c$. It is both bounded and bounded away from zero on metric balls in the former, while in the latter, metric balls may contain curves of arbitrarily short curve length.
This prevents the use of Sobolev interpolations to obtain uniform estimates on metric balls in the case of open curves. 

Nevertheless, the following were established in \cite{BMM23}:
\begin{lemma}[{\cite[Lemma~6.6]{BMM23}}]
    \label{lemm: ellc lipschitz}
    Let $n\ge2$, $d\ge1$, and $G$ be a metric of form \eqref{eq: metric} with , possibly, $a_2=0$. Then the map $c\mapsto\ell_c^{\nicefrac{3}{2}}$ is Lipschitz continuous on metric balls of $\left(\Ind,\dist\right)$. Moreover, the Lipschitz constant in $B_r(c_0)$ depends only on $r$ and $\ell_{c_0}$, and on a given metric ball $B_r(c_0)\subset\Ind$, the Lipchitz constant of $c\mapsto\ell_c^{3/2}$ can be bounded by  
    \[
     L=C(1+r)(1+2r)e^{2Cr}\left(1+\ell_{c_0}^{3/2}\right),
    \]
where $C=C(a_0,a_n,d,n)$. 

\end{lemma} 
\begin{cor}
    \label{cor: boundness of ellc}
    The function $c\mapsto\ell_c$ is bounded from above on any metric ball in $\Ind$. Moreover, for any $c\in\Ind$, there exists some $r>0$ so that $\ell_c$ is also bounded away from zero on the metric ball $B_r(c)$.
\end{cor}

This local boundedness of curve-length is, in turn, used to prove the following:

\begin{theorem}[{\cite[Theorem~6.3]{BMM23}}]
    \label{thm: non vanishing converge}
    Let $n\ge2$, $d\ge1$, and $G$ be a constant coefficient Riemannian metric on $\Ind$ as in \eqref{eq: metric} with, possibly, $a_2=0$. Then any Cauchy sequence of immersions $c_m\in\Ind$ for $m=1,2,...$, such that $\ell_{c_m}\ge\delta>0$ is bounded away from zero, converges to some $c_\infty\in\Ind.$
\end{theorem}

By \Cref{thm: non vanishing converge}, to characterize the metric completion of $\Ind$, it is sufficient to examine sequences of immersions with vanishing curve length.

\section{Proof of the main result}
\label{sec: results}
We now proceed to prove \Cref{thm: main}. We set $d=1$ and consider the metric completion of $\left(\In,\dist\right)$, for $n\ge2$. In this setting, an immersion $c\in\In$ can be uniquely described by a parametrization, length, and starting point by 
\begin{equation}
    \label{eq: decomposition}
c\left(\theta\right)=r+\ell\cdot\varphi\left(\theta\right),\quad \ell>0,\ r\in\R, \ \varphi\in\diff.\end{equation}
In particular, we note that the inclusion  map 
\begin{equation}
    \label{eq: identity map def}
    \iota: \left(\diff,\|\cdot\|_{H^n}\right)\hookrightarrow\left(\In,\dist\right), \qquad \iota(\varphi)=\varphi
\end{equation}
is a homeomorphism onto its image by \Cref{lemm: topo equivalence}: Indeed, $\varphi_m\to \varphi$ in $\diff$, if and only if it converges in $H^n(0,2\pi)$, which is equivalent to converging in the open subset $\left(\In,\|\cdot\|_{H^n}\right)$, which is again equivalent to converging in $\left(\In,\dist\right)$.
We also note that $\In$ consists of two isometric connected components, one of increasing functions and one of decreasing functions.
From this point onward, we restrict our analysis to the component of increasing functions. 
Accordingly, when addressing a diffeomorphism $\varphi\in\diff$, we will assume $\varphi$ is orientation preserving. All subsequent results hold analogously for the second connected component.

\subsection{Estimates}
\label{subsec: estimates}
In this section, we present all estimates and results used to derive \Cref{thm: main}.
We begin by recalling two estimates from the first author's thesis \cite{gelman2025characterization} on the length of paths defined by translation and linear shrinkage: 

\begin{lemma}[{\cite[Lemma~3.1]{gelman2025characterization}}]
\label{lemma: translation dist}
Let $n\ge2$ and let $G$ be a constant-coefficient Riemannian metric on $\In$ as in \eqref{eq: metric}, with possibly $a_2=0$. Then for any $c\in\In$ and any  $v\in\mathbb{R}$,
\begin{equation}
    \dist\left(c,c+v\right)\le |v|\cdot\sqrt{a_0\ell_c},
\end{equation}
 where $\ell_c$ is the length of $c$, and $a_0$ is the coefficient in \eqref{eq: metric}.
\end{lemma}

\begin{lemma}[{\cite[Lemma~3.2]{gelman2025characterization}}]
\label{lem: shrink dist}
    Let $n\ge2$ and let $G$ be a constant-coefficient Riemannian metric on $\In$ as in \eqref{eq: metric}, with possibly $a_2=0$. For any $c\in\In$, and any $0<\lambda<1$:
    \[\dist\left( c, \lambda c\right)< 2\sqrt{\ell_c}\left(\ell_c\sqrt{\frac{a_0}{27}}+\sqrt{a_1}\right),\]
    where $\ell_c$ is the curve-length of $c$, and $a_0 ,a_1$ are the coefficients in \eqref{eq: metric}.
    
    As a result, for  $C=\frac{4\pi\sqrt{a_0}}{3\sqrt{3}}+2\sqrt{a_1}$, and any $c\in\In$ with $\ell_c\le 2\pi$ and any $0<\lambda<1$ we have
    \[\dist\left( c, \lambda c\right)\le C\sqrt{\ell_c}.\]
    
\end{lemma}

Proving the two estimates above follows from a straightforward calculation of the lengths of the paths defined by $t\mapsto tc$ for $t\in\left[\lambda,1\right]$ and $t\mapsto c+tv$ for $t\in\left[0,1\right]$, respectively. 
For these paths, only the zeroth-order term in the first case and the zeroth- and first-order terms in the second case are non-zero.

The combination of the last two lemmas implies the following: 
\begin{cor}
\label{cor: uni-para Cauchy}
Let $n\ge2$, and let $\left(c_m\right)_{m=1}^\infty\subset\In$ be a sequence with a fixed parametrization $\varphi\in\diff$, i.e., 
\[c_m\left(\theta\right)\coloneqq r_m +\ell_m \varphi\left(\theta\right),\quad r_m\in\R,\ \ell_m>0.\]

Then $(c_m)_{m=1}^\infty$ is a divergent Cauchy sequence iff $\ell_m\rightarrow0$. In this case, the limit point in $\overline{\left(\Id,\dist\right)}$ is independent of $\left(
\ell_m\right)_{m=1}^\infty\subset\R_{>0}$ and $\left(r_m\right)_{m=1}^\infty\subset\R$.
\end{cor}

\begin{proof}
First, assume that there is a subsequence $\ell_{m_k}\ge\delta>0$.
Then if $\left(c_m\right)_{m=1}^\infty$ is Cauchy, 
$c_{m_k}$ converges in $\In$ by \Cref{thm: non vanishing converge}. 
Thus the whole sequence also converges.

Conversely, assume $\ell_m\rightarrow0$. 
Denote $L_N\coloneqq\sqrt{\max_{m\ge N}\ell_m}$, and note that $L_N\rightarrow0$ as $N\rightarrow\infty$.
Fix $N$ so that $L_N<\sqrt{2\pi}$. For any pair $m,k\ge N$, we may choose $0<\ell_{m,k}<\min(\ell_m,\ell_k)\le L_N^2$ small enough so that
\[\ell_{m,k} |r_m|\le\ell_m,\ \ell_{m,k} |r_k|\le\ell_k,\] and define $\lambda_m \coloneqq \frac{\ell_{m,k}}{\ell_m}$ and  $\lambda_k \coloneqq\frac{\ell_{m,k}}{\ell_k}$.  \Cref{lem: shrink dist} implies that
\[\dist\left( c_m, \lambda_m c_m\right)< C\sqrt{\ell_m}\le CL_N, \quad \dist\left( c_k, \lambda_k c_k\right) < C\sqrt{\ell_k}\le CL_N,\]
where $C$ is independent of $N,m,k$. 
Since now $\ell_{\lambda_kc_k}=\ell_{m,k}=\ell_{\lambda_mc_m}$ and also  $\lambda_m |r_m|,\lambda_k |r_k|\le1$, \Cref{lemma: translation dist} implies 
\begin{equation*}
    \begin{split}
         \dist\left(\lambda_m c_m,\lambda_m c_m -\left(r_m\lambda_m\right) \right)& \le \sqrt{a_0\ell_{m,k}} \le \sqrt{a_0} L_N \\
         \dist\left(\lambda_k c_k,\lambda_k c_k -\left(r_k\lambda_k\right) \right)&  \le \sqrt{a_0\ell_{m,k}} \le \sqrt{a_0} L_N,
    \end{split}
\end{equation*}

Noting that $\lambda_m c_m -\left(r_m\lambda_m\right)=\lambda_k c_k -\left(r_k\lambda_k\right)=\ell_{m,k}\varphi$, we conclude \[\dist\left(c_m,c_k\right)\rightarrow0\quad \text{as} \quad N\rightarrow\infty.\] 

To verify that the limit point is independent of $\left(r_m\right)_{m=1}^\infty\subset\R$ and $\left(
\ell_m\right)_{m=1}^\infty\subset\R_{>0}$ (as long as it converges to zero), note  the claim can be applied to a sequence alternating between two different choices of $\left(
\ell_m\right)_{m=1}^\infty$ and $\left(r_m\right)_{m=1}^\infty$.
\end{proof}

In view of \Cref{cor: uni-para Cauchy}, we consider the mapping defined by
\begin{equation}
    \label{eq: homeorphism def}
    \begin{split}
    \mathcal{F}:\diff&\rightarrow\overline{\In}\setminus\In \\
    \varphi&\mapsto\lim_{n\rightarrow\infty}\frac{1}{n}\varphi.
    \end{split}
\end{equation}

$\mathcal{F}$ is well defined by the result of \Cref{cor: uni-para Cauchy}.
In the following section, we show that $\mathcal{F}$ is injective and continuous, and, subsequently, that $\mathcal{F}$ is invertible with a Lipschitz continuous inverse, concluding the proof of \Cref{thm: main}. Prior to presenting these proofs, however, we will first collect a few additional lemmas.

We first recall the following Poincar\'e estimate:
\begin{lemma}
\label{lemm: poincare est}
     For $n\ge2$, $c\in\In$  with $\ell_c=2\pi$, and $h\in T_c\In$, 
     satisfying $h(0)=h(2\pi)=0$, then 
            \[\frac{1}{2}\|h\|_{L^2(ds)}\le\|\partial_sh\|_{L^2(ds)}\le\|\partial_s^2h\|_{L^2(ds)}.\]
\end{lemma}
\begin{proof}
    By reparametrization invariance and the fact that $\ell_c=2\pi$, we may assume $c$ is arc-length-parametrized, in which case $\partial_s=\partial_\theta$ and $ds=d\theta$. Notice that both the boundary conditions are preserved by reparameterization. In these settings, we notice $\int_0^{2\pi}\partial_\theta hd\theta=0$, and thus, the estimates follow from application of Wirtinger's inequalities given in \cite[Theorems~257-8]{hardy1952inequalities} to $h$ and $\partial_s h=\partial_\theta h$, respectively.
\end{proof}

In addition, we note the following:

\begin{lemma}
\label{lem: path-breakdown}
    For $n\ge2$, if $\gamma:[0,1]\rightarrow\In$ is $C^1$ with respect to  $t$, then we can write
    \[\gamma(t)=r(t)+\ell(t)\varphi(t),\] 
    where $r(t)\in\R$, $\ell(t)>0$, $\varphi(t)\in\diff$ for all $t\in[0,1]$, and are all $C^1$ with respect to  $t$. Moreover, considering $\varphi$ as a path $[0,1]\rightarrow\In$, 
    \begin{equation}
        \label{eq: diffgroup norm ineq}
    \|\partial_s^i\gamma_t(t)\|^2_{L^2(ds)}=\frac{1}{\left(\ell(t)\right)^{2i-3}}\|\partial_s^i\varphi_t(t)\|^2_{L^2(ds)},\quad \forall i=2,3...,n.
    \end{equation}
\end{lemma}

\begin{proof}
    Note that, if we show $t\mapsto\gamma(t,0)$ and $t\mapsto\ell_{\gamma(t)}$ are $C^1$ with respect to $t$, then it follows that $\varphi(t)\coloneqq\frac{1}{\ell_{\gamma(t)}}[\gamma(t,\cdot)-\gamma(t,0)]$ is $C^1$ as well.
    
    Fix $t_0\in[0,1]$. 
    By \Cref{cor: boundness of ellc}, for some $r>0$ we have that $c\mapsto\ell_c$ is bounded from above and away from zero on $B_r(\gamma(t_0))$. Thus, as noted previously,  
    \[
    m\|h\|_{\dot{H}^i(d\theta)}\le\|\partial_s^i h\|_{L^2(ds)}=\frac{1}{\ell_c^{(2n-3)\slash2}}\|h\|_{\dot{H}^i(d\theta)}\le M\|h\|_{\dot{H}^i(d\theta)}
    \]
    for some $0<m\le M<\infty$. 
    Therefore $\gamma$ is  $C^1$ as a map $[0,1]\rightarrow H^n\left([0,2\pi],\R^d\right)$. 
    By Sobolev inequalities, e.g., \cite[Theorem~7.40]{leoni2009first},  we have $\|c\|_{C^1}\le C\|c\|_{H^n(d\theta)}$, hence $c\mapsto c(0)$ is a bounded linear operator $H^n\left([0,2\pi],\R^d\right)\rightarrow\R$. As a result, the composition $t\mapsto\gamma(t,0)$ is differentiable.
    Similarly, $c\mapsto c'$ is also a bounded linear operator $H^n\left([0,2\pi],\R^d\right)\rightarrow C^0\left([0,2\pi]\right)$ and since $\gamma'(t,\theta)\ne0$ for all $t,\theta$, the composition $t\mapsto\int_0^{2\pi}|\gamma'(t,\theta)|d\theta$ is differentiable as well. 

    Lastly, \eqref{eq: diffgroup norm ineq} follows from the fact that $r(t),\ell(t)$ are independent of  $\theta$, and so
    \[
    \partial_s^2\left(\gamma_t\right)=\frac{1}{|\ell\varphi'|}\partial_\theta\left(\frac{\ell_t\varphi'+\ell\varphi_t'}{|\ell\varphi'|}\right)=\frac{1}{\ell\varphi'}\partial_\theta\left(\frac{\varphi_t'}{\varphi'}\right)=\ell^{-1}\partial_s^2\varphi_t,
    \]
    where $\partial_s$ in the rightmost expression is taken at $T_{\varphi(t)}\In$.
    The equality for higher-order terms then follows directly.
\end{proof}

We now prove the main lemma, which will then be used in \Cref{subsec: results} to establish \Cref{thm: main}.

\begin{lemma}
\label{lem: Constant dist inequality}
    Let $n\ge2$, and let $G$ be a constant-coefficient Riemannian metric as in \eqref{eq: metric} on $\In$. Recall the map $\mathcal{F}:\diff\rightarrow\overline{\In}$ defined in \eqref{eq: homeorphism def}, and the map $\iota:\diff\rightarrow\In$ defined in \eqref{eq: identity map def}. Then exist constants $r_0>0$ and $\tilde C=\tilde C(a_0,...,a_n,n)$ such that for any $\varphi,\psi\in\diff$, then $\dist\left(\mathcal{F}(\varphi),\mathcal{F}(\psi)\right)<r_0$ implies 
    \[\dist\left(\iota(\varphi),\iota(\psi)\right)\le \tilde{C}\dist\left(\mathcal{F}(\varphi),\mathcal{F}(\psi)\right).\]
\end{lemma}

\begin{remark*}
    In fact, if $n>3$, the constant $\tilde C$ can be taken to be independent of $a_3,...,a_{n-1}$. This follows by  deriving direct inequalities, as \eqref{eq: gamma-nterm inverse}, for all intermediate terms.
\end{remark*}

\begin{proof}
    We set $r_0>0$ so that $Cr_0(1+r_0)(1+2r_0)e^{2Cr_0}<1$, 
    where $C$ is the constant from \Cref{lemm: ellc lipschitz}. Given $\varphi_0,\varphi_1\in\diff$ with 
    \[\dist\left(\mathcal{F}(\varphi_0),\mathcal{F}(\varphi_1)\right)<r_0,\] we fix $0<\varepsilon<r_0-\dist\left(\mathcal{F}(\varphi_0),\mathcal{F}(\varphi_1)\right)$. By definition, $\lim_{m\rightarrow\infty}\dist\left(\frac{1}{m}\varphi_0,\frac{1}{m}\varphi_1\right)<r_0$. 
    We fix $N\in\mathbb{N}$ so that both
    \begin{equation}
        \label{eq: r0 lip bound}
        \frac{(2\pi)^{3\slash2}}{N^{3\slash2}}+Cr_0(1+r_0)(1+2r_0)e^{2Cr_0}\left(1+\frac{{(2\pi)}^{3/2}}{N^{3\slash2}}\right)<1,
    \end{equation}
    and $\dist\left(\frac{1}{m}\varphi_0,\frac{1}{m}\varphi_1\right)<r_0$ for all $m\ge N$.
    
    We now consider, for $m\ge N$, the metric ball $B_{r_0}(\frac{1}{m}\varphi_0)$. 
    By \Cref{lemm: ellc lipschitz}, we have  
    \[
    \left|\frac{(2\pi)^{3/2}}{m^{3\slash2}}-\ell_c^{3\slash2}\right|<Cr_0(1+r_0)(1+2r_0)e^{2Cr_0}\left(1+\frac{{(2\pi)^{3/2}}}{m^{3\slash2}}\right)\quad \forall c\in B_{r_0}(\frac{1}{m}\varphi_0).
    \]
    This implies, by \eqref{eq: r0 lip bound}, that $c\mapsto\ell_c$ is bounded by $1$ on $B_{r_0}(\frac{1}{m}\varphi_0)$.
    Next, we fix a path, $\gamma^m:[0,1]\rightarrow\In$, connecting $\frac{1}{m}\varphi_0$ and $\frac{1}{m}\varphi_1$, so that 
    \[\dist\left(\frac{1}{m}\varphi_0,\frac{1}{m}\varphi_1\right)<\len\left(\gamma^m\right)<\dist\left(\frac{1}{m}\varphi_0,\frac{1}{m}\varphi_1\right)+\varepsilon<r_0.\]
    Thus, $\gamma^m$ is contained in $B_{r_0}\left(\frac{1}{m}\varphi_0\right)$, and so $\ell_{\gamma(t)}\le1$ for all $t\in[0,1]$.   
    
    Considering \Cref{lem: path-breakdown}, we write 
    \[\gamma^m(t)=r^m(t)+\ell^m(t)\varphi^m(t),\] 
    where $r^m(t)\in\R$, $\ell^m(t)>0$, $\varphi^m(t)\in\diff$ for all $t\in[0,1]$, and are all piecewise-$C^1$ with respect to  $t$. As established in \eqref{eq: diffgroup norm ineq}, we have
    \[\|\partial_s^i\gamma^m_t(t)\|^2_{L^2(ds)}=\frac{1}{\left(\ell^m(t)\right)^{2i-3}}\|\partial_s^i\varphi^m_t(t)\|^2_{L^2(ds)},\quad \forall i=2,3...,n,\]
    where in the right hand side,  $\varphi^m_t\in T_{\varphi^m(t)}\In$.
    Since $\ell^m(t)\le1$, we have
    \begin{equation}
        \label{eq: gamma-twoterm inverse}
        \begin{split}
            \varepsilon+\dist\left(\frac{1}{m }\varphi_0,\frac{1}{m}\varphi_1\right)\ge&\len\left(\gamma^m\right)\\
            \ge&\sqrt{a_2}\int_0^1\|\partial_s^2\gamma^m_t(t)\|_{L^2(ds)}dt\\
            \ge&\sqrt{a_2}\int_0^1\|\partial_s^2\varphi^m_t(t)\|_{L^2(ds)}.
        \end{split}
    \end{equation}
    Similarly,
    \begin{equation}
        \label{eq: gamma-nterm inverse}
        \begin{split}
            \varepsilon+\dist\left(\frac{1}{m }\varphi_0,\frac{1}{m}\varphi_1\right)\ge&\len\left(\gamma^m\right)\\
            \ge&\sqrt{a_n}\int_0^1\|\partial_s^n\gamma^m_t(t)\|_{L^2(ds)}dt\\
            \ge&\sqrt{a_n}\int_0^1\|\partial_s^n\varphi^m_t(t)\|_{L^2(ds)}dt.
        \end{split}
    \end{equation}
    Noting that $\varphi^m(t)\in\diff$ for all $t$, we have $\varphi^m_t(0,t)=\varphi^m_t(2\pi,t)\equiv0$, which by \Cref{lemm: poincare est} implies 
    \[\|\partial_s^2\varphi^m_t\|_{L^2(ds)}\ge\frac{1}{2}\|\varphi^m_t\|_{L^2(ds)}.\]
    Thus, from \eqref{eq: gamma-twoterm inverse} and \eqref{eq: gamma-nterm inverse}, we have 
    \begin{equation*}
        \begin{split}
            \dist\left(\frac{1}{m }\varphi_0,\frac{1}{m}\varphi_1\right)\ge& \frac{1}{2}\int_0^1\sqrt{\frac{a_2}{4}}\|\varphi^m_t\|_{L^2(ds)}+\sqrt{a_n}\|\partial_s^n\varphi^m_t(t)\|_{L^2(ds)}dt\\
            \ge&  C_0\int_0^1\sqrt{\|\varphi^m_t(t)\|^2_{L^2(ds)}+\|\partial_s^n\varphi^m_t(t)\|^2_{L^2(ds)}}dt,
         \end{split}
    \end{equation*}
    where $C_0=C_0(a_0,a_n)>0.$
    Recall also that, since $\ell_{\varphi^m(t)}=2\pi$ for all $m,k>M$ and $t\in[0,1]$, by \Cref{lemm: norm est}, for some  $C_1=C_1(a_0,...,a_n,n)>0$, we have 
    \[\sum_{i=1}^na_i\|\partial_s^i\varphi^m_t\|_{L^2(ds)}^2\le C_1\left(\|\varphi^m_t(t)\|^2_{L^2(ds)}+\|\partial_s^n\varphi^m_t(t)\|^2_{L^2(ds)}\right).\]

    Thus, considering $\varphi^m$ is a path $[0,1]\rightarrow\iota\left(\diff\right)\subset\In$ connecting $\iota(\varphi_0)\text{ and }\iota(\varphi_1)$, we have
     \begin{equation}
        \varepsilon+\dist\left(\frac{1}{m }\varphi_0,\frac{1}{m}\varphi_1\right)\ge\tilde{C}^{-1}\len(\varphi^m_t)\ge \tilde{C}^{-1}\dist\left(\iota(\varphi_0),\iota(\varphi_1)\right),
    \end{equation}
where $\tilde C=\frac{\sqrt C_1}{C_0}$. Taking $m\rightarrow\infty$ we conclude
\[\varepsilon+\dist\left(\mathcal{F}(\varphi_0),\mathcal{F}(\varphi_1)\right)\ge\tilde{C}^{-1}\dist\left(\iota(\varphi_0),\iota(\varphi_1)\right),\] and since $\varepsilon$ is arbitrary, the assertion follows .
\end{proof}

\subsection{Proof of main theorem}
\label{subsec: results}

Next, we proceed to prove \Cref{thm: main}. We prove the injectivity of $\mathcal{F}$ in \Cref{prop: distinct limit points}, continuity in \Cref{prop: H continuous},  and surjectivity in \Cref{prop: char of comp}. Subsequently, we conclude $\mathcal{F}^{-1}$  is Lipschitz continuous in \Cref{prop: H Lip continuous}, thus, in particular, establishing that $\mathcal{F}$ is a homeomorphism.

\begin{prop}
   \label{prop: distinct limit points}
    Let $n\ge2$, and let $G$ be a constant-coefficient Riemannian metric as in \eqref{eq: metric} on $\In$.  Then the map $\mathcal{F}$ defined in \eqref{eq: homeorphism def} is injective.
\end{prop}

\Cref{prop: distinct limit points} is a direct consequence of \Cref{lem: Constant dist inequality}:

\begin{proof}
    If $\varphi\ne\psi\in\diff$, then 
    \[
    \dist\left(\iota(\varphi),\iota(\psi)\right)>0,
    \]
    since $\dist$ separates points.
    By \Cref{lem: Constant dist inequality}, either \[\dist\left(\mathcal{F}(\varphi),\mathcal{F}(\psi)\right)>r_0,\] 
    or 
    \[0<\dist\left(\iota(\varphi),\iota(\psi)\right)\le \tilde{C}\dist\left(\mathcal{F}(\varphi),\mathcal{F}(\psi)\right),\]
    concluding the proof.
\end{proof}

\begin{prop}
    \label{prop: H continuous}
    Let $n\ge2$, let $G$ be a constant coefficient Riemannian metric  on $\In$ as in \eqref{eq: metric}, possibly with $a_2=0$. Then the map $\mathcal{F}$ defined in $\eqref{eq: homeorphism def}$ is continuous.
\end{prop}

\begin{remark*}
    Note that, by \Cref{lemm: topo equivalence}, the topology induced on $\diff$ by $G$, and the topology induced by $\|\cdot\|_{H^n(d\theta)}$ are identical. 
    Hence, the continuity of $\mathcal{F}$ is equivalent for both choices of metrics.
\end{remark*}

\begin{proof}
    To show continuity, assume $\left(\varphi_m\right)_{m=1}^\infty\subset\diff$  converges to $\varphi_\infty\in\diff$, i.e., $\|\varphi_n-\varphi_\infty\|_{H^n(d\theta)}\rightarrow0$.
    By homogeneity of $\|\cdot\|_{H^n}$,
    \[\|\lambda\varphi_n-\lambda\varphi_\infty\|_{H^n(d\theta)}\rightarrow0\qquad \text{for any }\lambda>0.\]
    Note that $\lambda\varphi_n,\lambda\varphi_\infty\in\In$ for all $\lambda>0$, and thus, by \Cref{lemm: topo equivalence}, 
    $\dist\left(\lambda\varphi_m,\lambda\varphi_\infty\right)\rightarrow0$ for all $\lambda>0$.

    Next, we show $\mathcal{F}(\varphi_m)\rightarrow\mathcal{F}(\varphi_\infty)$ with respect to the metric induced on $\overline\In$ by $\dist$. As in any metric space, it is enough to show that any subsequence has a further subsequence that converges to $\mathcal{F}(\varphi_\infty)$. 

    For simplicity of notation, assume $\mathcal{F}(\varphi_m)$ denotes a chosen subsequence. By \Cref{lem: shrink dist},  it follows that, for all $m\in \mathbb{N}\cup\{\infty\}$,  
    \[\dist\left(\mathcal{F}(\varphi_m),\lambda\varphi_m\right)\le C\sqrt{\lambda}\qquad \text{for any }\lambda\in(0,1),\]
    where $C=C(a_0,a_1)$ is independent of $m$.
    Set $m_1=1$, for each $k$ we choose $m_{k}>m_{k-1}$ large enough so that 
    \[\dist\left(\frac{1}{k^2}\varphi_l,\frac{1}{k^2}\varphi_\infty\right)<\frac{1}{k}\qquad \text{for all } l\ge m_k.\]
    Together, these assure 
        \begin{equation*}
            \begin{split}
                \dist\bigg(\mathcal{F}(\varphi_{m_k}),\mathcal{F}(\varphi_\infty)\bigg)&\le
                \dist\left(\mathcal{F}(\varphi_{m_k}),\frac{1}{k^2}\varphi_{m_k}\right)\\
                &+\dist\left(\frac{1}{k^2}\varphi_{m_k},\frac{1}{k^2}\varphi_\infty\right)
                +\dist\left(\frac{1}{k^2}\varphi_\infty,\mathcal{F}(\varphi_\infty)\right) \\
                &\le\frac{1+2C}{k},
            \end{split}
        \end{equation*}
where we used that $\dist\left(\frac{1}{k^2}\varphi_m,\frac{1}{k^2}\varphi_\infty\right)=\lim_{l\rightarrow\infty}\dist\left(\frac{1}{k^2}\varphi_m,\frac{1}{k^2}\varphi_l\right)
\le \frac{1}{k}$. 
Hence, $\mathcal{F}(\varphi_{m_k})\to \mathcal{F}(\varphi_\infty)$, which completes the proof.
\end{proof}

\begin{prop}
\label{prop: char of comp}
       Let $n\ge2$, and let $G$ be a constant-coefficient Riemannian metric as in \eqref{eq: metric} on $\In$.  Then the map $\mathcal{F}$ defined in \eqref{eq: homeorphism def} is surjective.
    
\end{prop}

\begin{proof}
    Let $\eta\in\overline{\In}\setminus\In$. As noted in \Cref{cor: uni-para Cauchy}, we may assume it is the limit point of a sequence of the form $\frac{1}{m} \varphi_m$ for $m=1,2,...$, where $\varphi_m\in\diff$. 
    Our objective is to show $\eta=\mathcal{F}(\varphi_\infty)$ for some $\varphi_\infty\in\diff$.

    First, we note that \Cref{lem: shrink dist} implies $\dist\left(\mathcal{F}(\varphi_m),\lambda\varphi_m\right)\le C\sqrt{\lambda}$, for all $\lambda\le2\pi$ and $m\in\mathbb{N}$, where $C=C(a_0,a_1)$. 
    Thus, by triangle inequality, $\dist(\mathcal{F}(\varphi_m),\eta)\rightarrow0$. 
    In particular, we may choose $N$ large enough, so that $\dist\left(\mathcal{F}(\varphi_m),\mathcal{F}(\varphi_k)\right)<r_0$ for all $m,k>N$, where $r_0$ is the universal radius introduced in \Cref{lem: Constant dist inequality}.

    We then conclude, by \Cref{lem: Constant dist inequality}, that
     \begin{equation}
        \tilde{C}\dist\left(\mathcal{F}(\varphi_m),\mathcal{F}(\varphi_k)\right)\ge \dist\left(\iota(\varphi_m),\iota(\varphi_k)\right),
    \end{equation}
    where $\tilde C=\tilde C(a_0,...,a_n,n)$. 
    Since $\mathcal{F}\left(\varphi_m\right)$ is convergent by assumption, we conclude $\left(\iota(\varphi_m)\right)^\infty_{m=1}$ is Cauchy. 
    Note that $\ell_{\iota(\varphi_m)}\equiv2\pi$, and so, by \Cref{thm: non vanishing converge}, $\iota(\varphi_m)$ converges with respect to $\dist$ to some immersion $\varphi_\infty\in\In$. Furthermore, by \Cref{lemm: topo equivalence},  $\|\varphi_m-\varphi_\infty\|_{H^n}\rightarrow0$, and by  standard Sobolev embedding, $\|\varphi_m-\varphi_\infty\|_{C^1}\rightarrow0$ as well. 
    This assures $\varphi_\infty(0)=0,\varphi_\infty(2\pi)=2\pi$,  and together with the fact $\varphi_\infty$ is an immersion, we conclude $\varphi_\infty$ is a diffeomorphism.

    Having established that $\iota(\varphi_m)\rightarrow\varphi_\infty=\iota(\varphi_\infty)$ with respect to $\dist$, \Cref{prop: H continuous} assures also $\dist\left(\mathcal{F}(\iota_m),\mathcal{F}(\iota_\infty)\right)\rightarrow0$,
    and so $\eta=\mathcal{F}(\varphi_\infty)$, as desired.
\end{proof}

Finally, we show that $\mathcal{F}^{-1}$ is Lipschitz continuous:

\begin{prop}
    \label{prop: H Lip continuous}
     Let $n\ge2$, and let $G$ be a constant-coefficient Riemannian metric as in \eqref{eq: metric} on $\In$.  Then the map \[\mathcal{F}^{-1}:\left(\overline{\In}\setminus\In,\dist\right)\rightarrow\left(\iota\left(\diff\right),\dist\right)\] is Lipchitz continuous. 
\end{prop}

\begin{proof}
    Let $\varphi,\psi\in\diff$. By \Cref{lem: Constant dist inequality}, if $\dist\left(\mathcal{F}(\varphi),\mathcal{F}(\psi)\right)<r_0$, then 
    \[\dist\left(\iota(\varphi),\iota(\psi)\right)\le \tilde{C}\dist\left(\mathcal{F}(\varphi),\mathcal{F}(\psi)\right).\]
 
    On the other hand, if $\dist\left(\mathcal{F}(\varphi),\mathcal{F}(\psi)\right)\ge r_0$. We notice that 
    \begin{equation}
        \begin{split}
        \dist\left(\iota(\varphi),\iota(\psi)\right)&\le \dist\left(\iota(\varphi),\mathcal{F}\left(\varphi\right)\right)\\
        &\quad +\dist\left(\mathcal{F}(\varphi),\mathcal{F}(\psi)\right)+\dist\left(\iota(\psi),\mathcal{F}\left(\psi\right)\right).
        \end{split}
    \end{equation}
By \Cref{lem: shrink dist}, 
$\dist\left(\iota(\psi),\mathcal{F}(\psi)\right),\dist\left(\iota(\varphi),\mathcal{F}(\varphi)\right)\le C$, for $C=C(a_0,a_1)$.
Thus,
\[\dist\left(\iota(\varphi),\iota(\psi)\right)\le\left(\frac{2C}{r_0}+1\right)\dist\left(\mathcal{F}(\varphi),\mathcal{F}(\psi)\right),\]
where we used the fact $\dist\left(\mathcal{F}(\varphi),\mathcal{F}(\psi)\right)\ge r_0$. 
Taking $L\coloneqq\max\left(\left(\frac{2C}{r_0}+1\right),\tilde{C}\right)$ concludes the proof.
\end{proof}

\bibliographystyle{abbrv} 
\bibliography{thesis} 

\end{document}